\documentclass[ejs,preprint]{imsart}
\RequirePackage[OT1]{fontenc}
\RequirePackage{amsthm,amsmath}
\startlocaldefs
\numberwithin{equation}{section}
\theoremstyle{plain}

\endlocaldefs
\usepackage{amsmath}
\usepackage{graphicx}
\usepackage{amsfonts}
\usepackage{amssymb}%
\providecommand{\U}[1]{\protect\rule{.1in}{.1in}}
\usepackage{color}

\newtheorem{theorem}{Theorem}[section]

\newtheorem{algorithm}[theorem]{Algorithm}

\newtheorem{definition}[theorem]{Definition}

\newtheorem{proposition}[theorem]{Proposition}
\newtheorem{remark}{Remark}[section]

\begin{document}
\begin{frontmatter}
\title{Bootstrap uniform central limit theorems for~Harris recurrent Markov chains \thanksref{T1}}
\thankstext{T1}{Author is a beneficiary of the French Government scholarship Bourse Eiffel, managed by Campus France.}
\runtitle{Bootstrap uniform CLTs for Harris recurrent Markov chains}
\author{\fnms{Gabriela} \snm{Cio\l{}ek}\ead[label=e1]{gabrielaciolek@gmail.com}}
\address{AGH University of Science and Technology,\\ al. Mickiewicza 30, 30-059 Krakow, Poland \\ Modal'X, Universit\'{e} Paris Ouest Nanterre la D\'{e}fense, \\200 Avenue de la R\'{e}publique, 92000 Nanterre, France \\ \printead{e1}}
\runauthor{G. Cio\l{}ek}
\begin{keyword}[class=MSC]
\kwd[Primary ]{62G09}
\kwd[; secondary ]{62G20}
\kwd{60J05}
\end{keyword}
\begin{keyword} \kwd{Bootstrap, Markov chains, regenerative processes,
Nummelin splitting technique, empirical processes indexed by classes of
functions, entropy, robustness, Fr\'{e}chet differentiability }\end{keyword}
\begin{abstract}
 The~main objective of this paper is to establish bootstrap uniform functional central limit theorem for~Harris recurrent Markov chains over uniformly bounded classes of~functions. We show that~the~result can be generalized also to~the~unbounded case. To~avoid some complicated mixing conditions, we make use of the well-known regeneration properties of~Markov chains. We show that in~the~atomic case the~proof of~the~bootstrap uniform central limit theorem for Markov chains for functions dominated by a function in~$L^{2}$ space proposed by Radulovi\'{c} (2004) can be significantly simplified. Finally, we prove bootstrap uniform central limit theorems for Fr\'{e}chet differentiable functionals in a Markovian setting. 
\end{abstract}
\end{frontmatter}

\section{Introduction}

\pagenumbering{arabic}
\setcounter{page}{1}%
The~naive bootstrap for indentically distributed and independent random variables introduced by Efron (1979) has gradually evolved and new types of~bootstrap schemes in both i.i.d. and dependent setting were established. A~detailed review of various bootstrap methods such as moving block bootstrap (MBB), nonoverlapping block bootstrap (NBB) or cilcular block bootstrap (SBB) for dependent data can be found in Lahiri (2003). The main idea of~block bootstrap procedures is to resample blocks of observations in~order to~capture the dependence structure of~the~original sample. However, as~indicated by many authors, these procedures struggle with many problems. For instance, popular MBB method requires the~stationarity for observations that usually results in~failure of this method in non-stationary setting (see Lahiri (2003) for~more details). Furthermore, the asymptotic behaviour of~MBB method is highly dependent on~the~estimation of the bias and~of~the~asymptotic variance of~the~statistic of~interest that is a~significant drawback when considering practical applications. Finally, it is noteworthy, that the rate of convergence of~the~MBB distribution is slower than that's of~bootstrap distribution in~the~i.i.d. setting. Moreover, all mentioned block bootstrap procedures struggle with~the problem of the choice of the length of~the~blocks of~data in~order to~reflect the~dependence structure of~the~original sample. 
\par It is rather surprising that the bootstrap theory for Markov chains has been paid relatively limited attention given~the~extensive investigation and~development of various bootstrap methods for both i.i.d. and dependent data. One of~the~first bootstrap results for Markov chains was obtained by Datta and McCormick (1993). The~proposed method relies on~the renewal properties of Markov chains when a (recurrent) state is visited infinitely often. The~idea behind the~procedure is to~resample a~deterministic number of~data blocks which~are corresponding to regeneration cycles. However, the~method proposed by Datta and McCormick is not second-order correct. Bertail and Cl\'{e}men\c{c}on (2007) have proposed the~modification of~this~procedure which~gives the~second-order correctness in~the~stationary case, but fails in~the~nonstationary setting.  Bertail and Cl\'{e}men\c{c}on (2006) have proposed two effective methods for~bootstrapping Markov chains: the Regenerative block bootstrap (RBB) method for~atomic chains and~~Approximate block bootstrap method (ARBB) for~general Harris recurrent Markov chains. The~main idea behind these~procedures is to mimick the~renewal (pseudo-renewal in general Harris case) structure of the chain by~drawing regeneration data blocks, until the~length of the reconstructed bootstrap sample is larger than the length of~the~original data. Blocks before the first and after the last regeneration times are discarded in order to avoid large bias. In the atomic setting,  the~RBB method has the~uniform rate of~convergence  of~order~$O_{\mathbb{P}}(n^{-1})$ which is the~optimal rate of convergence in~the~i.i.d. case. Bertail and~Cl\'{e}men\c{c}on~(2006) have proved the~second-order correctness of~the~ARBB procedure in~the~unstudentized stationary case, the rate of convergence is close to that in the i.i.d. case. It is noteworthy that for both~methods, the~division of the data into blocks is completely data-driven what is a significant advantage in comparison to block bootstrap methods. It is worthy of mention, that in parallel to the paper of Bertail and~Cl\'{e}men\c{c}on~(2006), the Markov chains~bootstrap CLT for the mean under~no~additional assumptions was proposed by~Radulovi\'{c} (2004). 
\par Bootstrap results for Markov chains established by Radulovi\'{c} (2004) and Bertail and Cl\'{e}men\c{c}on (2006)  allow naturally to extend~the bootstrap theory to~empirical processes indexed by~classes of~functions in a~Markovian setting. Radulovi\'{c} (2004) has proved the~bootstrap uniform functional central limit theorem over uniformly bounded classes of functions~$\mathcal{F}$. In~mentioned paper, Radulovi\'{c} considers countable regenerative Markov chains and indicates that with additional uniform entropy condition the bootstrap result can be extended to~the~uncountable case. Gorst-Rasmussen and B\o{}gsted (2009) have proved the~bootstrap uniform central limit theorem over~classes of~functions whose~envelope is in~$L^{2}.$ They have considered regenerative case which~was motivated by~their study of~queuing systems with~abandonment.
\par This paper generalizes the Radulovi\'{c}'s (2004) bootstrap result for~empirical processes for~Markov chains. We establish the bootstrap uniform functional central limit theorem over a~permissible uniformly bounded classes of functions in~general Harris case. We also show that by arguments of~Tsai (1998), the~condition of~the~uniform boundedness of~$\mathcal{F}$ can be weakened and~it is sufficient to~require only that~$\mathcal{F}$ has an~envelope~$F$ in~$L^{2}.$
The proof of~the~bootstrap uniform CLT for Harris recurrent Markov chains is closely related to~the~uniform CLT for~countable atomic Markov chains proposed by~Radulovi\'{c}. Similarly as in~his paper, the main struggle is the random number of pseudo-regeneration blocks. However, using regeneration properties of~Markov chains, it is possible to replace the random number of blocks with~its deterministic equivalent what~simplifies the analysis of~asymptotic properties of~the~studied empirical process. The~arguments from the~proof of~main theorem of~this~paper can be also applied directly to~the proof of bootstrap uniform CLT for atomic Markov chains proposed by~Radulovi\'{c}~(2004). Thus, we can significantly simplify the proof of~the~Radulovi\'{c}'s result and~apply standard probability inequalities for i.i.d. blocks of data to~show the~asymptotic stochastic equicontinuity of~the bootstrap version of~original empirical process indexed by~uniformly bounded class of function.
\par Regenerative properties of Markov chains can be applied in order to extend some concepts in robust statistics from i.i.d. to a~Markovian setting. Martin and Yohai (1986) have shown that, generally, proving that statistics are robust in dependent case is a~challanging task. Bertail and Cl\'{e}men\c{c}on (2006) have defined an influence function and Fr\'{e}chet differentiability on the torus what allowed to extend the notion of robustness from single observations to the blocks of data instead. As shown in~Bertail and Cl\'{e}men\c{c}on (2015), this approach leads directly to~central limit theorems  (and their bootstrap versions) for Fr\'{e}chet differentiable functionals in a~Markovian setting. In our framework, we use the bootstrap asymptotic results for empirical processes indexed by~classes of functions to derive bootstrap uniform central limit theorems for~Fr\'{e}chet differentiable functionals in~a~Markovian case. Interestingly, there is no need to~consider blocks of data as in Bertail and Cl\'{e}men\c{c}on (2015).  We show that the theorems work when~classes of functions are permissible and~uniformly bounded, however, it is easy to weaken the last assumption and impose that~$\mathcal{F}$ has an~envelope in~$L^{2}.$
\par The paper is organized as follows. In section \ref{intro}, we introduce~the notation and~preliminary assumptions for~Markov chains. In~section~\ref{boothar}, we recall briefly some bootstrap methods for Harris recurrent Markov chains and formulate further necessary assumptions for the considered Markov chains. In~section~\ref{main}, we establish the~bootstrap uniform central limit theorem for~Markov chains. We give a~proof for~uniformly bounded classes of~functions and~show how~the~theory can be~easily extended to~the~unbounded case. We indicate that~using regeneration properties of~Markov chains, the~proof~of~uniform bootstrap central limit theorem for~countable chains proposed by~Radulovi\'{c} can be simplified. In section \ref{frech}, the bootstrap uniform central limit theorems for~Fr\'{e}chet differentiable functionals in a Markovian setting are established. We prove that the~central limit theorem holds when~classes of functions are uniformly bounded. Next, we generalize the theory to the unbounded case demanding that~$\mathcal{F}$ has an~envelope in~$L^{2}.$ In the last section, we enclose small appendix with a~proof of the~interesting property used in~the proofs of main asymptotic theorems in the previous section.

\section{Preliminaries}\label{intro}

\par We begin by introducing some notation and recall the~key concepts of the~Markov chains theory (see Meyn \& Tweedie (1996) for a detailed review and~references). For the~reader's convenience we keep our~notation in~agreement with~notation set in~Bertail and Cl\'{e}men\c{c}on (2006). All along this section $\mathbb{I}_{A}$ is the indicator function of the event $A.$

\par Let $ X = (X_{n})_{n\in \mathbb{N}}$ be a homogeneous Markov chain on~a~countably generated state space~$(E, \mathcal{E})$ with transition probability~$\Pi$ and~initial probability~$\nu.$
Note that for any~$B \in \mathcal{E}$ and~$n \in \mathbb{N},$ we have 
$$ X_{0} \sim \nu \text{\;\;and\;\;} \mathbb{P}(X_{n+1} \in B | X_{0}, \cdots, X_{n}) = \Pi(X_{n}, B) \text{\;\;a.s.}$$
In our framework, $\mathbb{P}_{x}$ (resp. $\mathbb{P}_{\nu})$ denotes the~probability measure such that $X_{0} = x$ and $ X_{0} \in E$ (resp. $X_{0} \sim \nu$), and~$\mathbb{E}_{x}\left(  \cdot \right)$ is~the~$\mathbb{P}_{x}$-expectation (resp. $\mathbb{E}_{\nu}\left(  \cdot \right) $ is the $\mathbb{P}_{\nu}$-expectation). In the following, we assume that~$X$ is $\psi$ -irreducible and aperiodic, unless it is specified otherwise.
\par We are particularly interested in the~atomic structure of Markov chains. It is shown by~Nummelin (1978) that any chain that possesses some recurrent properties can be extended to a chain which has an~atom.
\begin{definition}
Assume that $X$ is aperiodic and $\psi$-irreducible. We say that~a~set $A \in \mathcal{E}$ is~an~accessible atom if~for all $x,\;y \in A$ we have $\Pi(x,\cdot) = \Pi(y,\cdot)$ and~$\psi (A) >0.$ In that case we call $X$ atomic.
\end{definition}
\par In our framework, we are interested in the~asymptotic behaviour of positive recurrent Harris  Markov chains. We say~that $X$ is \textit{Harris recurrent} if starting from any~point~$x \in E$ and any set such that $\psi(A) > 0$, we have
$\mathbb{P}_{x}(\tau_{A} < +\infty) = 1.$
Observe that the~property of Harris recurrence ensures that~$X$ visits set~$A$ infinitely often a.s.. It follows directly from the~strong Markov property, that given any initial law $\nu,$  the sample paths can be divided into i.i.d. blocks corresponding to~the~consecutive visitis of~the~chain to~atom~$A.$ The segments of data are of the form:
$$ \mathcal{B}_{j}=(X_{1+\tau_{A}(j)}, \cdots, X_{\tau_{A}(j+1)}),\; j\geq 1$$
and take values in the torus $\cup_{k=1}^{\infty} E^{k}.$
 
\par We define the sequence of~regeneration times~$(\tau_{A}(j))_{j \geq 1}$. The~sequence consists of~the~successive points of~time when the~chain forgets its past. Let
$$ \tau_{A} = \tau_{A}(1) = \inf\{ n \geq 1: X_{n} \in A\}$$ 
be the first time when the chain hits the~regeneration set $A$
and
$$ \tau_{A}(j) = \inf\{ n > \tau_{A}(j-1), X_{n} \in A\} \text{\;for\;} j \geq 2. $$
\par In our framework, we consider steady-state behaviour of Markov chains. One of the crucial stability results of~interest is the Kac's theorem which enables to~write~functionals of~the~stationary distribution~$\mu$~as the functionals of distribution of~a~regenerative block. Indeed, for positive recurrent Markov chain if $\mathbb{E}_{A}(\tau_{A}) < \infty,$ then~the~unique invariant probability distribution $\mu$ is~the~Pitman's occupation measure given by
$$ \mu(B) = \frac{1}{\mathbb{E}_{A}(\tau_{A})} \left( \sum_{i=1}^{\tau_{A}} \mathbb{I} \{X_{i} \in B\}\right)\; \forall B \in \mathcal{E}.$$
\par We introduce few more pieces of notation: throughout the paper we write $l_{n} = \sum_{i=1}^{n} \mathbb{I}\{X_{i} \in A\}$ for the~total number of~consecutive visits of the chain to the atom~$A,$ thus we observe $l_{n}+1$ data blocks. We make the~convention that~$B^{(n)}_{l_{n}} = \emptyset$ when~$\tau_{A}(l_{n}) = n.$ Furthermore, we denote by~$l(B_{j}) = \tau_{A}(j+1) - \tau_{A}(j),\; j\geq 1,$ the~length of~regeneration blocks. Note that the~by~the~Kac's theorem we have that\;$ \mathbb{E}(l(B_{j})) = \mathbb{E}_{A}(\tau_{A}) = \frac{1}{\mu(A)}.$ Consider $\mu-$ integrable function $f:\;E \rightarrow \mathbb{R}.$ By~${u}_{n}(f) = \frac{1}{\tau_{A}(l_{n}) -\tau_{A}(1)}\sum_{i=1}^{n} f(X_{i})$ we denote the~estimator of~the~unknown asymptotic mean~$\mathbb{E}_{\mu}(f(X_{1})).$
\begin{remark}
In order to avoid large bias of the estimators based on the~regenerative blocks we discard the data before the first and after the last pseudo-regeneration times (for more details refer to~Bertail and Cl\'{e}men\c{c}on~(2006), page 693).
\end{remark}

\subsection{General Harris Markov chains and the splitting technique}

In this subsection, we recall the so-called \textit{splitting technique}
introduced in~Nummelin (1978). The~technique allows to
extend the~probabilistic structure of any Harris chain in order to~artificially construct a~regeneration set. In~the~following, unless specified otherwise, $X$ is a general, aperiodic, $\psi$-irreducible chain with transition kernel $\Pi$.
\begin{definition}
We say that a~set~$S \in \mathcal{E}$ is small if there exists a~parameter $\delta > 0,$ a~positive probability measure~$\Phi$ supported by~$S$ and an~integer $m \in N^{*}$ such that
\begin{equation}\label{eq:minorization}
 \forall x\in S,\; A \in \mathcal{E}\;\; \Pi^{m}(x,A) \geq \delta\;\Phi(A),
\end{equation}
where $\Pi^{m}$ denotes the $m$-th iterate of~the~transition probability~$\Pi.$
\end{definition}
\begin{remark}
It is noteworthy that in general case it is not obvious that small sets having positive irreducible measure  exist. Jain and Jamison (1967) showed that they do exist for any irreducible kernel~$\Pi$ under the assumption that the state space is countably generated.
\end{remark}

\par We expand the sample space in order to~define a sequence $(Y_{n})_{n\in
\mathbb{N}}$ of independent r.v.'s with parameter $\delta.$ We define a~joint
distribution $\mathbb{P}_{\nu,\mathcal{M}}$ of $X^{\mathcal{M}}=(X_{n}%
,Y_{n})_{n\in\mathbb{N}}$ . The construction relies on the mixture
representation of $\Pi$ on $S,$ namely $
\Pi(x,A)=\delta\Phi(A)+(1-\delta)\frac{\Pi(x,A)-\delta\Phi(A)}{1-\delta}.$
It can be retrieved by the following randomization of the transition
probability $\Pi$ each time the chain $X$ visits the set $S$. If $X_{n}\in S$ and

\begin{itemize}
\item if $Y_{n}=1$ (which happens with probability $\delta\in\left]
0,1\right[  $), then $X_{n+1}$ is distributed according to the probability
measure $\Phi$, 

\item if $Y_{n}=0$ (that happens with probability $1-\delta$), then $X_{n+1} $
is distributed according to the probability measure $(1-\delta)^{-1}(\Pi
(X_{n},\cdot)-\delta\Phi(\cdot)).$
\end{itemize}

\par This bivariate Markov chain $X^{\mathcal{M}}$ is called the \textit{split
chain}. It takes its values in $E\times\left\{  0,1\right\}  $ and possesses
an atom, namely $S\times\left\{  1\right\}  $. The split chain~$X^{\mathcal{M}}$ inherits all the stability and communication properties of~the chain~$X.$  The~regenerative blocks of~the split chain are i.i.d. (in~case m = 1 in \eqref{eq:minorization}). If the~chain $X$ satisfies $\mathcal{M}(m, S, \delta, \Phi)$ for~$m > 1,$ then~the~blocks of~data~are $1$-dependent, however, it is easy to adapt the theory from the~case when~$m=1$ (see for instance Levental (1988)).
\subsection{Regenerative blocks for dominated families}
Throughout the rest of the paper, the minorization condition $\mathcal{M}$ is fulfilled with
$m=1,$ unless specified otherwise. We assume that the family of the
conditional distributions $\{\Pi(x,dy)\}_{x\in E}$ and the initial
distribution $\nu$ are dominated by~a~$\sigma$-finite measure $\lambda$ of
reference, so that  $\nu(dy)=f(y)\lambda(dy)$ and  $\Pi(x,dy)=p(x,y)\lambda
(dy)$, for all $x\in E.$ The~minorization condition requests that
$\Phi$ is absolutely continuous with respect to $\lambda$  and
that$\;p(x,y)\geq\delta\phi(y),$ $\lambda(dy)$ a.s. for any $x\in S,$ with
$\Phi(dy)=\phi(y)dy$. Consider the binary random sequence $Y$ constructed via
the Nummelin's technique from the parameters inherited from condition $\mathcal{M}$. We want to~approximate the Nummelin's construction. Note that the distribution of $Y^{(n)}=(Y_{1},...,$ $Y_{n})$
conditionally to $X^{(n+1)}=(x_{1},...,x_{n+1}) $ is the tensor product of
Bernoulli distributions given by: for all $\beta^{(n)}=(\beta_{1}%
,...,\beta_{n})\in\left\{  0,1\right\}  ^{n},$ $x^{(n+1)}=(x_{1}%
,...,x_{n+1})\in E^{n+1},$
\[
\mathbb{P}_{\nu}\left(  Y^{(n)}=\beta^{(n)}\mid X^{(n+1)}=x^{(n+1)}\right)
=\prod_{i=1}^{n}\mathbb{P}_{\nu}(Y_{i}=\beta_{i}\mid X_{i}=x_{i},\text{
}X_{i+1}=x_{i+1}),
\]
with, for $1\leqslant i\leqslant n,$ 
\begin{itemize}
\item if $x_{i}\notin S,$ $\mathbb{P}_{\nu}(Y_{i}=1\mid X_{i}=x_{i},$
$X_{i+1}=x_{i+1})=\delta,$

\item if $x_{i}\in S,\ \mathbb{P}_{\nu}(Y_{i}=1\mid X_{i}=x_{i},$
$X_{i+1}=x_{i+1})=\delta\phi(x_{i+1})/p(x_{i},x_{i+1}).$
\end{itemize}
\par Observe that conditioned on $X^{(n+1)}$, from $i=1$ to $n$, $Y_{i}$ is distributed according to~the Bernoulli distribution with parameter $\delta$, unless $X$ has
hit the small set $S$ at time $i$: then, $Y_{i}$ is drawn from~the~Bernoulli distribution with parameter $\delta\phi(X_{i+1})/p(X_{i},X_{i+1}).$
We denote by~$\mathcal{L}^{(n)}(p,S,\delta,\phi,x^{(n+1)})$ this probability
distribution. If we were able to generate $Y_{1},...,$ $Y_{n}$, so that
$X^{\mathcal{M}(n)}=\left(  (X_{1},Y_{1}),...,(X_{n},Y_{n})\right)  $
be a realization of the split chain $X^{\mathcal{M}},$
then we would be able to do the block decomposition of~the~sample path $X^{\mathcal{M}(n)}
$ leading to~asymptotically i.i.d. blocks. Note, that in the above procedure the knowledge about the~transition density~$p(x,y)$ is required in~order to~generate random variables~$(Y_{1}, \cdots, Y_{n}).$ To~deal with~this~problem in~practice, Bertail and Cl\'{e}men\c{c}on~(2006) proposed the~approximating construction of~the~above~procedure. We proceed as follows. We construct an~estimator $p_{n}(x,y)$ of~$p(x,y)$ based on~$X^{(n+1)}$ (and~$p_{n}(x,y)$ satisfies $p_{n}(x,y) \geq \delta \phi(y),\; \lambda(dy)-$a.s. and~$p_{n}(x,y)>0, 1\leq i \leq n).$ Next, we generate random vector $\hat{\mathbb{Y}}_{n} = (\hat{Y}_{1}, \cdots, \hat{Y}_{n})$ conditionally to~$X^{(n+1)}$ from~distribution~$\mathcal{L}^{(n)}(p_{n}, S, \delta, \gamma, X^{(n+1)})$ which~is an~approximation of~the~conditional distribution~$\mathcal{L}^{(n)}(p, S, \delta, \gamma,X^{(n+1)})$ of~$(Y_{1}, \cdots, Y_{n})$ for~given~$X^{(n+1)}.$ 
\par In this setting, we define the~successive hitting times of~$A_{\mathcal{M}} = S \times \{1\}$ as~$\hat{\tau}_{A_{\mathcal{M}}}(i),\; i=1,\cdots, \hat{l}_{n},$ where $\hat{l}_{n} = \sum_{i=1}^{n} \mathbb{I} \{ X_{i} \in S, \hat{Y}_{i} = 1\}$ is the~total number of visits of the split chain to~$A_{\mathcal{M}}$ up to time~$n.$ The approximated blocks are of the form:
\begin{align*} \hat{\mathcal{B}}_{0} = (X_{1}, \cdots, X_{\hat{\tau}_{A_{\mathcal{M}}}(1)}), \cdots, \hat{\mathcal{B}}_{j} = (X_{\hat{\tau}_{A_{\mathcal{M}}}(j)+1}, \cdots, X_{\hat{\tau}_{A_{\mathcal{M}}}(j+1)}), \cdots,
\\
 \hat{\mathcal{B}}_{\hat{l}_{n}-1} = (X_{\hat{\tau}_{A_{\mathcal{M}}}(\hat{l
}_{n}-1)+1}, \cdots, X_{\hat{\tau}_{A_{\mathcal{M}}}(\hat{l}_{n})}),\;\hat{\mathcal{B}}_{\hat{l}_{n}}^{(n)} = (X_{\hat{\tau}_{A_{\mathcal{M}}}(\hat{l}_{n})+1}, \cdots, X_{n+1}).
\end{align*}
Moreover, we denote by $\hat{n}_{A_{\mathcal{M}}} = \hat{\tau}_{A_{\mathcal{M}}}(\hat{l}_{n}) - \hat{\tau}_{A_{\mathcal{M}}}(1) = \sum_{i=1}^{\hat{l}_{n}-1} l(\hat{B}_{j})$
the total~number of~observations after the first and before the~last pseudo-regeneration times.
Let $$\sigma^{2}_{f} = \frac{1}{\mathbb{E}_{A_{\mathcal{M}}}(\tau_{A_{\mathcal{M}}})}\mathbb{E}_{A_{\mathcal{M}}} \left( \sum_{i=1}^{\tau_{A_{\mathcal{M}}}} \{f(X_{i}) - \mu(f) \}^{2}\right)$$ be the~asymptotic variance. 
Furthermore, we set that
$$ \hat{\mu}_{n}(f) = \frac{1}{\hat{n}_{A_{\mathcal{M}}}}\sum_{i=1}^{\hat{l}_{n}-1} f(\hat{B}_{j}), \text{\;\; where\;\;} f(\hat{B}_{j}) = \sum_{i= 1 + \hat{\tau}_{A_{\mathcal{M}}}(j)}^{\hat{\tau}_{A_{\mathcal{M}}}(j+1)} f(X_{i})$$
and
$$ \hat{\sigma}_{n}^{2}(f) = \frac{1}{\hat{n}_{A_{\mathcal{M}}}}   \sum_{i=1}^{\hat{l}_{n}-1} \left\{f(\hat{B}_{i}) - \hat{\mu}_{n}(f)l(\hat{B}_{i})\right\}^{2}.$$
\par We briefly indicate that there exists a~connection between $\alpha$- mixing coefficients and~regeneration times for~Harris recurrent Markov chains.
The~strong~$\alpha$-mixing coefficient between $\sigma$-fields $\mathcal{A}$ and~$\mathcal{B}$ is defined  as
$$ \alpha(\mathcal{A}, \mathcal{B}):= \sup_{(A,B) \in \mathcal{A} \times \mathcal{B}} | \mathbb{P}(A \cap B) - \mathbb{P}(A) \mathbb{P}(B)|.$$
The~strong mixing coefficients related to~a~sequence of~random variables are defined by
$$ \alpha(k) = \sup_{n} \sup_{ A \in \hat{\xi}_{n}} \sup_{B \in \hat{\xi}^{n}} | \mathbb{P}_{\mu}(A \cap B) - \mathbb{P}_{\mu}(A) \mathbb{P}_{\mu}(B)|,$$
where $\hat{\xi}_{n} = \sigma(X_{i}, i \leq n)$ and $ \hat{\xi}^{n} = \sigma(X_{i}, i \geq n).$
By Theorem~2 from~Bolthausen (1982), we know that for~stationary Harris chains if~for~some~$\lambda \geq 0$ the sum $\sum_{m} m^{\lambda} \alpha(m) < \infty,$ then for~all~$B \in \mathcal{E}$ such that $\mu(B) > 0$ we have~$\mathbb{E}_{\mu} (\tau_{B}^{1 + \lambda}) < \infty,$ where~$\tau_{B} = \inf\{ n \geq 1: X_{n} \in~B\}.$
This result guarantees that the~rate of~decay of~strong~mixing coefficients is polynomial. This is a weaker condition, because~usually the~exponential rate of~decay is assumed.
\section{Bootstrap methods for~Harris recurrent Markov chains}\label{boothar}
In this section we recall shortly some bootstrap methods for Harris recurrent Markov chains which are essential to establish our bootstrap versions of uniform central limit theorems for Markov chains. We formulate necessary assumptions which must be satisfied by the chain in order to our theory could work.
\subsection{ARBB method}
In this subsection we recall the~Approximate block bootstrap algorithm (ARBB) introduced by~Bertail and~Cl\'{e}men\c{c}on (2006). The~ARBB method allows to~utilize the~pseudo-regeneration structure of~the~split chain in~order to~generate the~bootstrap blocks $ B_{1}^{*}, \cdots, B_{k}^{*}$ which are obtained by~resampling pseudo-regeneration data blocks~$\hat{B}_{1}, \cdots, \hat{B}_{\hat{l}_{n}-1}.$ The algorithm allows to compute the estimate of~the~sample distribution of~some statistic~$T_{n} = T(\hat{B}_{1}, \cdots, \hat{B}_{\hat{l}_{n}-1})$ with standarization $S_{n} = S(\hat{B}_{1}, \cdots, \hat{B}_{\hat{l}_{n}-1}).$ For the sake of clarity, we recall the~ARBB bootstrap procedure below. The~algorithm proceeds as follows:
\begin{algorithm}[ARBB procedure] \label{ARBB} 
\begin{enumerate}
\item Draw sequentially bootstrap data blocks $ B_{1}^{*}, \cdots, B_{k}^{*}$ (we denote the~length of~the~blocks by~$l(B^{*}_{j}),\; j=1, \cdots, k$) independently from the empirical distribution function
$$ \hat{\mathcal{L}}_{n} = \frac{1}{\hat{l}_{n}-1} \sum_{i=1}^{\hat{l}_{n}-1} \delta_{\hat{B}_{i}},$$ where~$\hat{B}_{i},\;i= 1,\cdots, \hat{l}_{n}-1$ are initial pseudo-regeneration blocks. We generate the~bootstrap blocks until the~joint length of~the~bootstrap blocks~$l^{*}(k) = \sum_{i=1}^{k} l(B_{i}^{*})$  exceeds~$n.$ We set $\l_{n}^{*} = \inf\{ k: l^{*}(k) > n\}.$
\item Bind the~bootstrap blocks from the step 1 and~construct the~ARBB bootstrap sample $ X^{*(n)} = (X_{1}^{*}, \cdots, X^{*}_{l_{n}^{*}-1}).$
\item Compute the~ARBB statistic and its~ARBB distribution, namely
$ T^{*}_{n} = T(X^{*(n)}) = T(B_{1}^{*}, \cdots, B^{*}_{l_{n}^{*}-1})$ and its standarization  $S^{*}_{n} = S(X^{*(n)})= S(B_{1}^{*}, \cdots, B^{*}_{l_{n}^{*}-1}).$
\item The~ARBB distribution is given by
$$ H_{ARBB}(x) = \mathbb{P}^{*}(S_{n}^{
*-1}(T_{n}^{*} - T_{n}) \leq x), $$
where~$\mathbb{P}^{*}$ is the~conditional probability given~the~data.
\end{enumerate}
\end{algorithm}
\par We introduce few more pieces of~notation. We denote by
$$ n^{*}_{A_{\mathcal{M}}} = \sum_{i=1}^{l_{n}^{*}-1} l(B_{j}^{*})$$ the length of~the~bootstrap sample, 
$$ \mu_{n}^{*}(f) = \frac{1}{n^{*}_{A_{\mathcal{M}}}} \sum_{i=1}^{l_{n}^{*} - 1} f(B_{i}^{*}) \text{\;\;and\;\;} \sigma_{n}^{*2}(f) = \frac{1}{n^{*}_{A_{\mathcal{M}}}} \sum_{i=1}^{l_{n}^{*} - 1} \{ f(B_{i}^{*}) - \mu_{n}^{*}(f)l(B_{j}^{*})\}^{2}.$$
\subsection{Preliminary bootstrap results for Markov chains}
\par Let $(X_{n})$ be~a~positive recurrent Harris Markov chain and~$(\alpha_{n})_{n \in \mathbb{N}}$ be a~sequence of~nonnegative numbers that converges to~zero. We impose the~following assumptions on~the~chain (compare with~Bertail and Cl\'{e}men\c{c}on (2006), page 700):
\begin{enumerate}
    \item $S$ is chosen so that $\inf_{x \in S} \phi(x) > 0$.
\item Transition density $p$ is estimated by~$p_{n}$ at the~rate $ \alpha_{n}$ (usually we consider $\alpha_{n} = \frac{\log(n)}{n}$) for the~mean squared error (MSE) when error is measured by~the~$L^{\infty}$ loss over~$S^{2}.$
\end{enumerate}
\par Moreover, we assume the~following conditions (for~a~comprehensive treatment on~these assumptions the~interested reader may refer to~Bertail and Cl\'{e}men\c{c}on (2006)). Let $k \geq 2$ be a real number. \\
$\mathcal{H}_{1}(f, k, \nu).$ The small set~$S$ is such that
$$ \sup_{x \in S} \mathbb{E}_{x} \left[ \left(\sum_{i=1}^{\tau_{S}} |f(X_{i})|\right)^{k}\right] < \infty$$
and
$$ \mathbb{E}_{\nu} \left[ \left(\sum_{i=1}^{\tau_{S}} |f(X_{i})|\right)^{k}\right] < \infty. $$
$\mathcal{H}_{2}(k, \nu).$ The set $S$ is such $ \sup_{x \in S} \mathbb{E}_{x}(\tau_{S}^{k}) < \infty$ and~$\mathbb{E}_{\nu}(\tau_{S}^{k}) < \infty.$
\newline
$\mathcal{H}_{3}. \;\; p(x,y)$ is estimated by $p_{n}(x,y)$ at the rate~$\alpha_{n}$ for the~MSE when error is measured by~the~$L^{\infty}$ loss over~$ S\times S:$
$$ \mathbb{E}_{\nu} \left( \sup_{(x,y) \in S \times S} | p_{n}(x,y) - p(x,y)|^{2} \right) = O(\alpha_{n}), \text{\;\; as\;} n \rightarrow \infty.$$
$\mathcal{H}_{4}.$ The density $\phi$ is such that $\inf_{x \in S} \phi(x) > 0.$
\newline
$\mathcal{H}_{5}.$ The transition density $p(x,y)$ and its estimate $p_{n}(x,y)$ are bounded by a constant $ R < \infty$ over~$S^{2}.$
\begin{remark}
In the following, we assume that~$\alpha_{n} \sim (\frac{\log(n)}{n})^{s/s+1}$ (see Bertail and Cl\'{e}men\c{c}on (2006) for more details).
\end{remark}
\par Before we establish main result of this paper we recall two theorems from~Bertail and Cl\'{e}men\c{c}on (2006) that essentially establish the~consistency of~ARBB procedure for~pseudo-regeneration blocks. 
\begin{theorem}\label{3.1}
Assume that the conditions $[1]$ and $[2]$ are satisfied by the~chain and $\mathcal{H}_{1}(f, \rho, \nu),\; \mathcal{H}_{2}(\rho, \nu)$ with~$ \rho \geq 4,\; \mathcal{H}_{3}, \mathcal{H}_{4}$ and~$\mathcal{H}_{5}$ hold. Then, as $n \rightarrow \infty$ we have
$$ \hat{\sigma}^{2}_{n}(f) \rightarrow \sigma^{2}_{f} \text{\;\; in\;} \mathbb{P}_{\nu}-\text{probability}$$
and
$$ \hat{n}^{1/2}_{A_{\mathcal{M}}}\frac{\hat{\mu}_{n}(f) - \mu(f)}{\hat{\sigma}_{n}(f)} \rightarrow \mathcal{N}(0,1) \text{\;\; in distribution under\;}\mathbb{P}_{\nu}.$$
\end{theorem}
\par Denote by $BL_{1}(\mathcal{F})$ the~set of~all~$1$-Lipschitz bounded functions on~$l^{\infty}(\mathcal{F}).$ We define~the~bounded Lipschitz metric on~$l^{\infty}(\mathcal{F})$ as
$$  d_{BL_{1}}(X,Y) = \sup_{b \in BL_{1}(l^\infty (\mathcal{F}))} | \mathbb{E} b(X) - \mathbb{E}b(Y)|; \;\; X, Y \in l^{\infty}(\mathcal{F}).$$
Convergence in~bounded Lipschitz metric is~correspondent to~weak convergence. Expectations of~nonmeasurable elements are understood as~outer expectations.
\begin{definition}\label{consistency}
We say that~$\mathbb{Z}_{n}^{*}$ is weakly consistent if
$ d_{BL_{1}}( \mathbb{Z}_{n}^{*}, \mathbb{Z}_{n}) \xrightarrow{P} 0.$
Analogously, $\mathbb{Z}_{n}^{*}$ is strongly consistent if
$ d_{BL_{1}}( \mathbb{Z}_{n}^{*}, \mathbb{Z}_{n}) \xrightarrow{\text{\;a.s.}} 0.$
\end{definition}
\begin{theorem}\label{3.2}
Under the~hypotheses of the~Theorem \ref{3.1} , we have the following convergence in~probability under~$\mathbb{P}_{\nu}$:
$$ \Delta_{n} = \sup_{x \in \mathbb{R}} | H_{ARBB}(x) - H_{\nu}(x)| \rightarrow 0, \text{\;\;as\;} n \rightarrow \infty, $$
where
$$ H_{\nu}(x) = \mathbb{P}_{\nu}(x) \left( \hat{n}_{A_{\mathcal{M}}}^{1/2}\sigma^{-1}_{f}( \hat{\mu}_{n}(f) - \mu(f)) \leq x\right)$$
and
$$ H_{ARBB}(x) = \mathbb{P}^{*} \left(n^{* 1/2}_{A_{\mathcal{M}}} \hat{\sigma}_{n}^{-1}(f) ( \mu_{n}^{*}(f) - \hat{\mu}_{n}(f)) \leq x| X^{(n+1)}\right).$$
\end{theorem}
\par In the following, the~convergence $X_{n} \xrightarrow{P^{*}} X$ in~$\mathbb{P}_{\nu-}$probability ($\mathbb{P}_{\nu}$ -a.s.) along the sample is understood~as
$$ \mathbb{P}^{*}(|X_{n} - X| > \epsilon|X^{(n+1)}) \xrightarrow{n \rightarrow \infty} 0 \text{\;\; in\;}\mathbb{P}_{\nu} -\text{probability\;\;}(\mathbb{P}_{\nu} \text{ -a.s.}).$$

\section{Uniform bootstrap central limit theorems for~Markov chains}\label{main}
To establish the uniform bootstrap CLT over~permissible, uniformly bounded classes of functions~$\mathcal{F},$ we need to be sure that the size~of~$\mathcal{F}$ is not too large (it is typical requirement when considering~uniform asymptotic results for empirical processes indexed by classes of functions). To control the size of~$\mathcal{F},$ we require the finiteness of~its \textit{covering number}~$N_{p}(\epsilon, Q, \mathcal{F})$ which is interpreted as the minimal number of balls with radius~$\epsilon$ needed to~cover~$\mathcal{F}$ in the norm~$L_{p}(Q)$ and~$Q$ is a measure on~$E$ with finite support. Moreover, we impose the~finiteness of~the~\textit{uniform entropy integral} of~$\mathcal{F},$ namely 
$$ \int_{0}^{\infty} \sqrt{\log N_{2}(\epsilon, \mathcal{F})} d\epsilon < \infty, \text{\;\; where\;\;} N_{2}(\epsilon, \mathcal{F}) = \sup_{Q} N_{2}(\epsilon, Q, \mathcal{F}).$$
\par For the sake of completeness we recall below Theorem 5.9 from~Levental (1988) which~is crucial to~establish uniform bootstrap CLT in~general Harris case.
\begin{theorem}\label{5.9}
Let $(X_{n})$ be a positive recurrent Harris chain taking values in~$(E, \mathcal{E}).$ Let~$\mu$ be the~invariant probability measure for~$(X_{n}).$
Assume further that~$\mathcal{F}$ is a~uniformly bounded class of~measurable ~functions on~$E$ and 
$$  \int_{0}^{\infty} \sqrt{\log N_{2}(\epsilon, \mathcal{F})} d\epsilon < \infty. $$ If $\sup_{x \in A}\mathbb{E}_{x}(\tau_{A})^{2 + \gamma} < \infty  \;\;(\gamma > 0\;\text{fixed}),$
where~$A$ is atomic set for the~chain, then the~empirical process $Z_{n}(f) = n^{1/2} (\mu_{n} - \mu)(f)$ converges weakly as a~random element of~$l^{\infty}(\mathcal{F})$ to a~gaussian process $G$ indexed by~$\mathcal{F}$ whose sample paths are bounded and~uniformly continuous with respect to~the~metric~$L_{2}(\mu).$
\end{theorem}
\subsection{Main asymptotic results}
In this subsection we establish the~bootstrap uniform central limit theorem over~permissible, uniformly bounded classes of~functions which~satisfy the~uniform entropy condition.
\begin{theorem}\label{uniformclt}
Suppose that~$(X_{n})$ is positive recurrent Harris Markov chain and the assumptions from the~Theorem \ref{3.2} are satisfied by~$(X_{n}).$ Assume further that $\mathcal{F}$ is a permissible, uniformly bounded class of functions and the following uniformity condition holds
\begin{equation}\label{eq:unicond}
\int_{0}^{\infty} \sqrt{\log N_{2}(\epsilon, \mathcal{F})} d\epsilon < \infty.
\end{equation}
Then the process 
\begin{equation}
\mathbb{Z}^{*}_{n} = n^{* 1/2}_{A_{\mathcal{M}}} \left[ \frac{1}{n^{*}_{A_{\mathcal{M}}}} \sum_{i=1}^{l_{n}^{*}-1} f(B_{i}^{*}) - \frac{1}{\hat{n}_{A_{\mathcal{M}}}} \sum_{i=1}^{\hat{l}_{n}-1} f(\hat{B}_{i})\right]
\end{equation}
converges in probability~under~$\mathbb{P}_{\nu}$ to~a~gaussian process $G$ indexed by~$\mathcal{F}$ whose sample paths are bounded and uniformly continuous with respect to the~metric~$L_{2}(\mu).$
\end{theorem}
\begin{proof}
The proof is based on the~bootstrap central limit theorem introduced by~Gin\'{e} and~Zinn (1990). To prove the weak convergence of~the~process $\mathbb{Z}_{n}^{*}$ we need to show 
\begin{enumerate}
    \item Finite dimensional convergence of~distributions of~$\mathbb{Z}_{n}^{*}$ to~$G$.
\item Stochastic asymptotic equicontinuity in~probability under~$\mathbb{P}_{\nu}$ with~respect to~the~totally bounded semimetric~$\rho$ on~$\mathcal{F}.$
\end{enumerate} 
\par Firstly, we prove that $ (\mathbb{Z}_{n}^{*}(f_{i1}), \cdots,\mathbb{Z}_{n}^{*}(f_{ik})) $
converges weakly in probability to $ (G(f_{i1}), \cdots, G(f_{ik}))$ for~every fixed finite collection of~functions $ \{f_{i1}, \cdots, f_{ik}\} \subset \mathcal{F}.$ Denote by~$\xrightarrow{L}$ the weak convergence in~law in~the sense of~Hoffmann-J\o{}rgensen. We want to show that for~any fixed collection $(a_{1}, \cdots, a_{k}) \in \mathbb{R}$ we have
$$ \sum_{j=1}^{k} a_{j}\mathbb{Z}_{n}^{*}(f_{ij}) \xrightarrow{L} \mathcal{N}(0, \gamma^{2}) \text{\;\; in probability under\;} \mathbb{P}_{\nu},$$
 where
$$ \gamma^{2} = \sum_{j=1}^{k} a_{j}^{2} Var(\mathbb{Z}_{n}(f_{ij})) + \sum_{s \neq r} a_{i}a_{j} Cov(\mathbb{Z}_{n}(f_{is}), \mathbb{Z}_{n}(f_{ir})). $$
Let $h= \sum_{j=1}^{k} a_{j}f_{ij}.$ By linearity of~$h$ and Theorem \ref{5.9} we conclude that
\begin{equation}\label{eq:fconv}
\mathbb{Z}_{n}(h) \xrightarrow{L} G(h).
\end{equation}
The above convergence of~$\mathbb{Z}_{n}(h)$ coupled with~the~Theorems \ref{3.1} and~\ref{3.2} guarantee that $\mathbb{Z}_{n}^{*}(h) \xrightarrow{L} G(h)$ in~probability under~$\mathbb{P}_{\nu}.$
Thus, the finite dimensional convergence for the~$\mathbb{Z}_{n}^{*}(f),\;f \in \mathcal{F}$ is established.
\par To verify $[2]$ we need to check if~for every~$\epsilon > 0$
\begin{equation}\label{eq:equi} \lim_{\delta \rightarrow 0}\limsup_{n \rightarrow \infty} \mathbb{P}^{*}(\|\mathbb{Z}_{n}^{*}\|_{\mathcal{F}_{\delta}} > \epsilon) = 0 \text{\;\;in probability under\;} \mathbb{P}_{\nu},
\end{equation}
where $\|R\|_{\mathcal{F}_{\delta}} := \sup\{ |R(f) - R(g)| : \rho(f,g) < \delta\}$ and~$R \in l^{\infty}(\mathcal{F}).$ Moreover, $\mathcal{F}$ must be totally bounded in~$L_{2}(\mu).$ In fact, the latter was shown by~Levental (1988). For the reader's convenience we repeat the~reasoning from~the~mentioned paper. 
\par Consider class of functions~$\mathcal{H} = \{|f-g|: f,g \in \mathcal{F}\}.$ Denote by~$Q_{n}$ the~$n$-th empirical measure of~an~i.i.d. process whose law is~$\mu.$ Using~basic properties of~covering numbers we obtain that~$ N_{1}(\epsilon, \mathcal{G}, Q_{n}) \leq \left(N_{2}\left(\frac{\epsilon}{2}, \mathcal{F}\right)\right)^{2} < \infty$ and~thus by~the~SLLN for~$Q_{n}$ (see Theorem 3.6 in~Levental (1988)) we have that $$\sup_{h \in \mathcal{H}} |(Q_{n} - \mu)(h)| \rightarrow 0 \text{\; a.s.}(\mu).$$ Since $\mathcal{F}$ is totally bounded in~$L_{1}(Q)$ for~every measure~$Q$ with~finite support it follows that~is totally bounded in~$L_{1}(\mu).$ Moreover, one can show that if~an~envelope~of~$\mathcal{F}$ is in~$L_{2}(\mu),$ then~$\mathcal{F}$ is totally bounded in~$L_{2}(\mu).$
\par In order to~show~\eqref{eq:equi}, firstly, we replace the~random numbers $n^{*}_{A_{\mathcal{M}}}$ and~$l^{*}_{n}$ by~their deterministic equivalents. By~the~same arguments as~in~the~proof of~the~Theorem~\ref{3.2} (see Bertail and Cl\'{e}men\c{c}on (2006), page 710 for details) we have the~following convergences
$$ \frac{l(B^{*}_{j})}{n}\xrightarrow{P^{*}} 0 \text{\;\;and\;\;}\frac{n^{*}_{A_{\mathcal{M}}}}{n} \xrightarrow{P^{*}} 1 $$ in $\mathbb{P}_{\nu}-$probability along the sample path as~$n \rightarrow \infty$
and
$$ \frac{l_{n}^{*}}{n} - \mathbb{E}_{A_{\mathcal{M}}}(\tau_{A_{\mathcal{M}}})^{-1} \xrightarrow{P^{*}} 0 $$ in $\mathbb{P}_{\nu}-$probability along the sample path as~$n \rightarrow \infty.$ Thus, we conclude that
\begin{align*}
\mathbb{Z}_{n}^{*}(f) &= \sqrt{ n^{*}_{A_{\mathcal{M}}}} \left[ \frac{1}{n^{*}_{A_{\mathcal{M}}}} \sum_{i=1}^{l_{n}^{*}-1} f(B_{i}^{*}) - \frac{1}{\hat{n}_{A_{\mathcal{M}}}} \sum_{i=1}^{\hat{l}_{n}-1} f(\hat{B}_{i})\right] \\
&= \frac{1}{\sqrt{n^{*}_{A_{\mathcal{M}}} }}\left[  \sum_{i=1}^{l_{n}^{*}-1}\left\{ f(B_{i}^{*}) - \hat{\mu}_{n}(f)l(B_{i}^{*})\right\}\right]\\
&= \frac{1}{\sqrt{n}} \left[\sum_{i=1}^{1+ \left\lfloor{ \frac{n}{\mathbb{E}_{A_{\mathcal{M}}}(\tau_{A})}}\right\rfloor} \{ f(B_{i}^{*}) - \hat{\mu}_{n}(f)l(B_{i}^{*})\}\right] + o_{\mathbb{P}^{*}}(1),
\end{align*}
where $\lfloor{ \;x\; \rfloor}$ is an integer part of~$x \in \mathbb{R}.$
The preceding reasoning allows us to switch to~the process
$$ \mathbb{U}_{n}^{*}(f) = \frac{1}{\sqrt{n}} \left[\sum_{i=1}^{1+ \left\lfloor{ \frac{n}{\mathbb{E}_{A_{\mathcal{M}}}(\tau_{A})}}\right\rfloor}  \{ f(B_{i}^{*}) - \hat{\mu}_{n}(f)l(B_{i}^{*})\}\right].$$
Observe, that $\{f(B_{i}^{*}) - \hat{\mu}_{n}(f)l(B_{i}^{*})\}_{i \geq 1} $ forms the~sequence of~i.i.d. random variables.
 \par Next, take $h = f-g.$ Denote by~$w(n) = 1+ \left\lfloor{ \frac{n}{\mathbb{E}_{A_{\mathcal{M}}}(\tau_{A})}}\right\rfloor $ and~$Y_{i} = l(B_{i}^{*}) - \hat{\mu}_{n}(f)l(B_{i}^{*}).$ We have the following inequality (by the fact that $Y_{i}$'s are i.i.d.)
\begin{align*}
\mathbb{P}^{*}(\|\mathbb{U}^{*}_{n}(h)\|_{\mathcal{F}_{\delta}} > \epsilon) &\leq w(n) \mathbb{P}^{*}\left(\frac{1}{\sqrt{n}}\|h(B^{*}_{1}) - l(B_{1}^{*})\hat{\mu}_{n,h}\|_{\mathcal{F}_{\delta}} > \epsilon \right).\end{align*}
The right hand side of the above inequality is bounded by
\begin{align*}  w(n) \mathbb{P}^{*}\left(\|h(B^{*}_{1})\|_{\mathcal{F}_{\delta}} > \frac{\sqrt{n}\epsilon}{2} \right) + w(n) \mathbb{P}^{*}\left( |l(B^{*}_{1})|\|\hat{\mu}_{n,h}\|_{\mathcal{F}_{\delta}} > \frac{\sqrt{n}\epsilon}{2} \right) = I + II.
\end{align*}
\par In the following, we investigate the~asymptotic behaviour of~$I$ and~$II.$ Some of~the reasoning relies on~the useful proposition from~Radulovi\'{c} (2004).
\begin{proposition}\label{radul}
For any random variable $W,$ such that $\mathbb{E} W^{2} < \infty,$ there exists a~positive increasing function~$\phi: \mathbb{R}^{+} \rightarrow \mathbb{R}^{+}$ such that
$$ \lim_{x \rightarrow \infty} \frac{\phi(x)}{x^{2}} = + \infty \text{\;\;and\;\;} \mathbb{E} \phi(W) < \infty.$$
\end{proposition}
\begin{remark}
The~sketch of the proof of~the~above Proposition is moved to~the~Appendix section.
\end{remark}
\par By Markov's inequality, we have that
$$w(n) \mathbb{P}^{*}\left(\|h(B^{*}_{1})\|_{\mathcal{F}_{\delta}} > \frac{\sqrt{n}\epsilon}{2} \right) \leq w(n)\frac{\mathbb{E}^{*}(\phi(\|h(B^{*}_{1})\|_{\mathcal{F}_{\delta}}))}{\phi(\frac{\sqrt{n}\epsilon}{2})}$$
By the~Proposition~\ref{radul} we conclude that
$$ \frac{w(n)}{\phi\left(\sqrt{n}/2\right)} = \frac{w(n)}{n} \cdot \frac{n}{\phi\left(\sqrt{n}/2\right)} \rightarrow 0 \text{\;\;a.s.}$$
since $\frac{w(n)}{n} \leq 1.$ Note also that
$$ \mathbb{E}^{*}(\phi(\|h(B^{*}_{1})\|_{\mathcal{F}_{\delta}})) \leq \mathbb{E}^{*}(\phi(|2F(B^{*}_{1})|)) <\infty \text{\;\;a.s.}$$
since $\mathcal{F}$ is uniformly bounded.
Thus,
$$w(n) \mathbb{P}^{*}\left(\|h(B^{*}_{1})\|_{\mathcal{F}_{\delta}} > \frac{\sqrt{n}\epsilon}{2} \right) \rightarrow 0 \text{\;a.s.}$$
\par Next, we investigate the~asymptotic behaviour of~$II.$ By Markov's inequality, we have
$$ w(n) \mathbb{P}^{*}\left( |l(B^{*}_{1})|\|\hat{\mu}_{n,h}\|_{\mathcal{F}_{\delta}} > \frac{\sqrt{n}\epsilon}{2} \right) \leq 4 w(n) \frac{\mathbb{E}^{*}(|l(B^{*}_{1})|)^{2} \|\hat{\mu}_{n,h}\|^{2}_{\mathcal{F}_{\delta}}}{n}.
$$
We know that $\frac{w(n)}{n} \leq 1$ and~$\|\hat{\mu}_{n,h}\|_{\mathcal{F}_{\delta}} \rightarrow 0$ in~$\mathbb{P}_{\nu}-$probability because of~the~stochastic equicontinuity of~the~original process~$\mathbb{Z}_{n}.$ Moreover, it is proven in~Bertail and Cl\'{e}men\c{c}on (2006) that
$$ \mathbb{E}^{*}\left(l(B^{*}_{1})^{2} | X^{(n+1)}\right) \rightarrow \mathbb{E}_{A_{\mathcal{M}}}(\tau^{2}_{A_{\mathcal{M}}}) < \infty$$ in $\mathbb{P}_{\nu}-$probability along the sample as $n \rightarrow \infty.$
Thus,
$$w(n) \mathbb{P}^{*}\left( |l(B^{*}_{1})|\|\hat{\mu}_{n,h}\|_{\mathcal{F}_{\delta}} > \frac{\sqrt{n}\epsilon}{2} \right) \rightarrow 0$$ in $\mathbb{P}_{\nu}-$probability along the sample as $n \rightarrow \infty.$
\par The above reasoning implies that \eqref{eq:equi} holds. We have checked that both conditions $[1]$ and~$[2]$ are satisfied by~$\mathbb{Z}_{n}^{*}.$ Thus, we can apply the bootstrap CLT proposed by Gin\'{e} and Zinn (1990) which yields the~desired result.
\end{proof}
\begin{remark}
Theorem \ref{uniformclt} is a generalization of the~Theorem 2.2 from Radulovi\'{c} (2004) for countable Markov chains. Note that the reasoning from~the proof of the above theorem can be directly applied to the proof of~Radulovi\'{c}'s result. The~part concerning the~proof of~the~asymptotic stochastic equicontinuity~of the bootstrap version of the empirical process indexed by uniformly bounded class of functions~$\mathcal{F}$ can be significantly simplified. As shown in the proof of the~Theorem~\ref{uniformclt}, we can switch from the process $\mathbb{Z}_{n}^{*}(f)_{f \in \mathcal{F}} := \sqrt{n^{*}} \{\mu_{n^{*}}(f) - \mu_{n_{A}}(f)\},$ where $n_{A} = \tau_{A}(l_{n}) - \tau_{A}$ to the~process $$\mathbb{U}_{n}^{*}(f) = \frac{1}{\sqrt{n}} \left[\sum_{i=1}^{1+ \left\lfloor{ \frac{n}{\mathbb{E}_{A}(\tau_{A})}}\right\rfloor}  \{ f(B_{i}^{*}) - \mu_{n_{A}}(f)l(B_{i}^{*})\}\right]$$ and~the standard probability inequalities applied to~the~i.i.d. blocks of data yield the result.  
\end{remark}
\par In the following, we show that we can weaken the~assumption of~uniform boundedness imposed on~the~class~$\mathcal{F}.$ By the results of~Tsai (1998), it is sufficient that~$\mathcal{F}$ has an~envelope in~$L_{2}(\mu),$ then the~uniform bootstrap central limit theorem holds.   
\begin{theorem}\label{uniformbootun}
Suppose that~$(X_{n})$ is positive recurrent Harris Markov chain and the assumptions from the~Theorem \ref{3.2} are satisfied by~$(X_{n}).$ Assume further that $\mathcal{F}$ is a~permissible class of functions and~such that the envelope~$F$ satisfies
\begin{equation}\label{eq:encon} \mathbb{E}_{A_{\mathcal{M}}} \left[ \sum_{\tau_{A_{\mathcal{M}}} < j \leq \tau_{A_{\mathcal{M}}}(2)} F(X_{j})\right]^{2 + \gamma} < \infty, \;\; \gamma > 0\;(\text{fixed}).
\end{equation}
Suppose, that the~following uniformity condition holds
\begin{equation}
\int_{0}^{\infty} \sqrt{\log N_{2}(\epsilon, \mathcal{F})} d\epsilon < \infty.
\end{equation}
Then the process 
\begin{equation}
\mathbb{Z}^{*}_{n} = n^{* 1/2}_{A_{\mathcal{M}}} \left[ \frac{1}{n^{*}_{A_{\mathcal{M}}}} \sum_{i=1}^{l_{n}^{*}-1} f(B_{i}^{*}) - \frac{1}{\hat{n}_{A_{\mathcal{M}}}} \sum_{i=1}^{\hat{l}_{n}-1} f(\hat{B}_{i})\right]
\end{equation}
converges in probability~under~$\mathbb{P}_{\nu}$ to~a~gaussian process $G$ indexed by~$\mathcal{F}$ whose sample paths are bounded and uniformly continuous with respect to the~metric~$L_{2}(\mu).$
\end{theorem}
\begin{proof}
The proof of Theorem \ref{uniformbootun} goes analogously to the proof of the~Theorem~\ref{uniformclt} with few natural modifications. We indicate the critical points where the~changes are necessary. The notation remains in the agreement with~the~previous theorem.
\begin{itemize}
\item Theorem 4.3 from Tsai (1998) establishes the weak convergence  $$\mathbb{Z}_{n}(h) \xrightarrow{L} G(h).$$ 
\item According to Bertail and Cl\'{e}men\c{c}on (2006), Theorem \ref{3.2} is also true when~$f$ is unbounded (see Bertail and Cl\'{e}men\c{c}on (2006), page 706 for~details). Thus, the finite dimensional convergence of distributions of~$\mathbb{Z}_{n}^{*}$ to the right gaussian process is ensured.
\item It is shown in Tsai (1988) that $\mathcal{F}$ is totally bounded in~$L_{2}(\mu)$ when $\mathcal{F}$ fulfills only the condition that the~envelope~$F$ is in~$L_{2}(\mu)$ (see Tsai (1998), page 9 for details).
\item The~finiteness of~the~$\mathbb{E}^{*}(\phi(|2F(B_{1}^{*})|))$ in~$\mathbb{P}_{v}-$probability along the sample as~$n \rightarrow \infty$ follows from the~Proposition~\ref{radul}. We know that if~the~condition~\eqref{eq:encon} on~the~envelope~$F$ holds, there~exists a~positive increasing function $\phi:\; \mathbb{R}^{+} \rightarrow \mathbb{R}^{+}$ such that
$$\phi(x) = x^{2 + \gamma} \text{\;\; and\;\;} \mathbb{E}_{A_{\mathcal{M}}} \left[ \sum_{\tau_{A_{\mathcal{M}}} < j \leq \tau_{A_{\mathcal{M}}}(2)} \phi(F(X_{j}))\right]^{2 + \gamma} < \infty.
$$
\end{itemize}
\end{proof}
\begin{remark}
It is noteworthy that the uniform central limit theorem for Harris recurrent Markov chains in Tsai (1998) holds with~the~weaker condition on~the~envelope~$F,$ i.e. 
$$ \mathbb{E}_{A_{\mathcal{M}}} \left[ \sum_{\tau_{A_{\mathcal{M}}} < j \leq \tau_{A_{\mathcal{M}}}(2)} F(X_{j})\right]^{2 } < \infty.$$
However, in the unbounded version of the bootstrap uniform central limit theorem for Harris Markov chains we need to~assume the~finiteness of the~$ (2+\gamma)$-th moment in order to show the~finiteness of~the~$\mathbb{E}^{*}(\phi(|2F(B_{1}^{*})|))$ in~$\mathbb{P}_{v}-$probability along the sample as~$n \rightarrow \infty.$ 
\end{remark}
\section{ Bootstrapping Fr\'{e}chet differentiable functionals}\label{frech}
Robust statistics provides tools to deal with data when we suspect that they include a small proportion of outliers. Robust statistical methods are applied to the solution of many problems such as estimation of regression parametres, estimation of scale and location. One of the key concepts of robust statistics when detecting the outliers in the data is an influence function. In the i.i.d. setting, the influence function measures the change in the value of some functional~$\phi(P)$ if we replace some infinitesimally small part of~$P$ by a~pointmass~$x$ (see van der Vaart (2000) for a detailed treatment of these issues in the i.i.d. framework). Generalizing the concepts of robustness and influence function into dependent case is a~challeging task (see Bertail and Cl\'{e}men\c{c}on (2015) and references therein). In a~Markovian setting, 
one can measure the influence of~(approximate) regeneration data blocks instead of single observations. The regenerative approach proposed by~Bertail and Cl\'{e}men\c{c}on (2015) naturally leads to~central limit and~convolution theorems. 
\par In our framework, we show how we can use results from the previous section to yield the bootstrap uniform central limit theorems for~general differentiable functionals over uniformly bounded classes (and with an~envelope in~$L_{2}(\mu))$ of~functions~$\mathcal{F}.$
\subsection{ Preliminary assumptions and remarks}
In robust statistics the influence function plays a crucial role to detect outliers in data. Functions and estimators which have an~unbounded influence function should be carefully investigated, because the small proportion of the~observations would have too much influence on the estimator.
\par Let's make our considerations rigorous. We denote by~$\mathcal{P}$ the set of~all probability measures on~$E.$ We keep the~notation  in~agreement with~notation introduced in~Bertail and~Cl\'{e}men\c{c}on (2015).
\par The classical definition of the influence function is provided below.
\begin{definition}
Let $(\vartheta, \|\cdot\|)$ be a~separable Banach space. Let~$T: \mathcal{P} \rightarrow \vartheta$ be a~functional on~$\mathcal{P}.$ If the limit
$$ \frac{T((1-t)\mu + t \delta_{x}) - T(\mu)}{t}, \text{\;\;\;\;as\;} t \rightarrow 0$$
is finite for all $\mu \in \mathcal{P}$ and for any~$x \in E,$ then we say that~the~influence function $T^{(1)}: \mathcal{P} \rightarrow \vartheta$ of the functional~$T$ is  well-defined and for all $x \in E$
$$ T^{(1)}(x, \mu) = \lim_{t \rightarrow 0} \frac{T((1-t)\mu + t \delta_{x}) - T(\mu)}{t}.$$
\end{definition}
\par In the following, we recall the definition of~Fr\'{e}chet derivative which is an~important concept in robust statistics. In particular, Fr\'{e}chet differentiability ensures the existence of~the~influence function. Let~$d$ be some metric on~$\mathcal{P}.$
\begin{definition}\label{frechet}
We say that~the~functional $T: \mathcal{P} \rightarrow \mathbb{R}$ is Fr\'{e}chet differentiable at~$\mu_{0} \in \mathcal{P}$ for~a~metric~$d,$ if there exists a~continuous linear operator $DT_{\mu_{0}}$ (from the set of~signed measures of~the~form~$\mu - \mu_{0}$ in~$(\vartheta, \|\cdot \|))$ and a~function~$\epsilon^{(1)}(\cdot,\mu_{0}): \mathbb{R} \rightarrow (\vartheta, \|\cdot \|),$ which is continuous at~$0$ and~$\epsilon^{(1)}(0,\mu_{0}) = 0$ such that
$$ \forall\; \mu \in \mathcal{P},\;\;T(\mu) - T(\mu_{0}) = DT_{\mu_{0}} (\mu - \mu_{0}) + R^{(1)}(\mu,\mu_{0}),$$
where $R^{(1)}(\mu, \mu_{0}) = d(\mu, \mu_{0}) \epsilon^{(1)}(d(\mu, \mu_{0}), \mu_{0}).$ Furthermore, we say that $T$ has an~influence function $T^{(1)}(\cdot, \mu_{0})$ if the following representation holds for~$DT_{\mu_{0}}:$
$$ \forall\;\mu_{0} \in \mathcal{P},\;\;\; DT_{\mu_{0}} (\mu - \mu_{0}) = \int_{E} T^{(1)}(x, \mu_{0})\mu(dx).$$
\end{definition}
\par In the context of empirical processes indexed by classes of functions, when one want to derive  the uniform central limit theorems for generally differentiable functionals the appropriate choice of metric is the crucial point. We need to choose the metric carefully in order to precisely control the distance $d(\mu_{n}, \mu)$ and the remainder~$ R^{(1)}(\mu_{n}, \mu).$ In our framework we have decided to work with a~generalization of~the~Kolmogorov's distance which is defined as follows:
\begin{definition}
Let~$\mathcal{H}$ be a class of real-valued functions (we do not impose the measurability condition  as one can work with outer measures and the~Hoffmann-J\o{}rgensen (1991) convergence). We define a distance
\begin{equation}\label{eq:metrich}
d_{\mathcal{H}}(P,Q):= \sup_{h \in \mathcal{H}} \left| \int h d(P-Q)\right|
\end{equation}
for any $P, Q \;\in\;\mathcal{P}.$
\end{definition}
\par The choice of metric defined in~\eqref{eq:metrich} is inspired by~the~arguments given by~Barbe and Bertail (1995) and Dudley (1990). Essentially, one may want to work with metric~$d_{\mathcal{H}}$ because it enables very precise control of the distance $d(\mu_{n}, \mu).$ Moreover, in many cases we can find a class of functions~$\mathcal{H},$ which makes the~functionals Fr\'{e}chet diferrentiable for~$d_{\mathcal{H}}.$ The latter is a~significant advantage since choice of~metric that guarantees Fr\'{e}chet differentiability of~functionals is usually challenging (see Barbe and Bertail (1995) and Dudley (1990) for an~extensive treatment on~this subject).
\par Note, that permissible, uniformly bounded (or with an envelope in~$L_{2}(\mu)$) classes of functions~$\mathcal{F}$ fulfill the~conditions imposed on the class~$\mathcal{H}.$ Thus, we can ease the notation and write $d_{\mathcal{F}}$ for the distance defined by~\eqref{eq:metrich}.
\subsection{Bootstrap uniform central limit theorems for Fr\'{e}chet differentiable functionals}
In this subsection, we show how the results from Levental (1988), Tsai (1998) and~from the previous section yield the uniform bootstrap central limit theorems for~Fr\'{e}chet differentiable functionals. Before we formulate the theorems, we briefly recall the notation. In general Harris case, $$\mu_{n}^{*} = \frac{1}{n^{*}_{A_{\mathcal{M}}}} \sum_{i=1}^{l_{n}^{*}-1}f(B_{i}^{*}) \text{\;\;and\;\;} \hat{\mu}_{n} = \frac{1}{\hat{n}_{A_{\mathcal{M}}}} \sum_{i=1}^{\hat{l}_{n}-1}f(\hat{B}_{i}),$$ where~$\hat{B}_{i},\; i = 1, \cdots, \hat{l}_{n}-1$ are pseudo-regeneration blocks. In regenerative case, the empirical mean is of~the~form $$\mu_{n} = \frac{1}{n_{A}} \sum_{i=1}^{l_{n}-1}f(B_{i}).$$
\par The crucial observation in~order to~establish the results is, that as long as we can control the distance~$d_{\mathcal{F}}(\mu_{n}^{*}, \hat{\mu}_{n})$ (we require it would be sufficiently small), we can control the remainder term~$R^{(1)}(\mu^{*}_{n}, \hat{\mu}_{n}).$ By the uniform central limit theorem, the~linear part of~the~$T(\mu_{n}^{*}) - T(\hat{\mu}_{n})$ is converging weakly to~a~desired gaussian process which yields our result.
\begin{theorem}\label{BootMark}
Let $\mathcal{F}$ be a~permissible, uniformly bounded class of functions, such that
$$ \int_{0}^{\infty} \sqrt{\log N_{2}(\epsilon, \mathcal{F})}d\epsilon < \infty.$$ Suppose that the conditions of Theorem~\ref{uniformclt} hold and~$T: \mathcal{P} \rightarrow \mathbb{R}$ is Fr\'{e}chet differentiable functional at~$\mu.$ Then, in general Harris positive recurrent case, we have that $ n^{1/2} (T(\mu_{n}^{*}) - T(\hat{\mu}_{n})) $ converges weakly in~$l^{\infty}(\mathcal{F})$ to~a~gaussian process~$G_{\mu}$ indexed by~$\mathcal{F},$ whose sample paths are bounded and uniformly continuous with respect to the metric~$L_{2}(\mu).$
\end{theorem}
\begin{remark}
It is obvious that the above theorem  works also in the~regenerative case. Replace~$\mathcal{A}_{\mathcal{M}}$ and the~ $\hat{\mu}_{n}$ for the split chain by~$\mathcal{A}$ and $\mu_{n}$ respectively. Then, under the assumptions from Theorem \ref{BootMark}, we have the weak convergence in~$l^{\infty}(\mathcal{F})$  to the  gaussian process indexed by~$\mathcal{F},$ whose sample paths are bounded and uniformly continuous with respect to the metric~$L_{2}(\mu).$
\end{remark}
\begin{proof}
Without loss of generality, we assume that~$\mathbb{E}_{\mu} T^{(1)}(x, \mu) = 0.$ By the Fr\'{e}chet differentiability formulated in definition~\ref{frechet} we have
\begin{equation}\label{eq:ffunctional}
T(\hat{\mu}_{n}) - T(\mu) = DT_{\mu}(\hat{\mu}_{n} - \mu) + d_{\mathcal{F}}(\hat{\mu}_{n}, \mu) \epsilon^{(1)}(d_{\mathcal{F}}(\hat{\mu}_{n}, \mu), \mu)
\end{equation}
and
\begin{equation}\label{eq:dfunctional}
 T(\mu_{n}^{*}) - T(\mu) = DT_{\mu}(\mu_{n}^{*} - \mu) + d_{\mathcal{F}}(\mu_{n}^{*}, \mu) \epsilon^{(1)}(d_{\mathcal{F}}(\mu_{n}^{*}, \mu), \mu).
\end{equation}
Thus,
\begin{align*}\label{eq:boott}
\sqrt{n}(T(\mu_{n}^{*}) - T(\hat{\mu}_{n})) &= \sqrt{n}\left(DT_{\hat{\mu}_{n}}(\mu_{n}^{*} - \hat{\mu}_{n})\right) + \sqrt{n}\left(d_{\mathcal{F}}(\hat{\mu}_{n}, \mu) \epsilon^{(1)}(d_{\mathcal{F}}(\hat{\mu}_{n}, \mu), \mu)\right) \\&\;\;\;+ \sqrt{n}\left(d_{\mathcal{F}}(\mu_{n}^{*}, \mu) \epsilon^{(1)}(d_{\mathcal{F}}(\mu_{n}^{*}, \mu), \mu)\right) .
\end{align*}
\par We show that~$d_{\mathcal{F}}(\hat{\mu}_{n}, \mu)$ and $d_{\mathcal{F}}(\mu_{n}^{*}, \mu)$ are of order $O_{\mathbb{P}_{\nu}}(n^{-1/2}).$ Theorem~\ref{5.9} guarantees that
$$ \sqrt{n} d_{\mathcal{F}}(\hat{\mu}_{n}, \mu) \xrightarrow{L} \sup_{f \in \mathcal{F}} |G(f)|, \text{\;\;as\;} n \rightarrow \infty,$$
where $G$ is gaussian process whose sample paths are bounded and uniformly continuous with respect to the metric $L_{2}(\mu).$
Thus, $d_{\mathcal{F}}(\hat{\mu}_{n}, \mu) = O_{\mathbb{P}_{\nu}}(n^{-1/2}).$ 
\par Next, observe that
$$ d_{\mathcal{F}}(\mu_{n}^{*}, \mu) \leq d_{\mathcal{F}}(\mu_{n}^{*}, \hat{\mu}_{n}) + d_{\mathcal{F}}(\hat{\mu}_{n}, \mu).$$
From the~Theorem~\ref{uniformclt} we conclude that
$$ \sqrt{n} d_{\mathcal{F}}(\mu_{n}^{*}, \hat{\mu}_{n}) \xrightarrow{L^{*}} \sup_{f \in \mathcal{F}} |G(f)|, \text{\;\;as\;} n \rightarrow \infty.$$
Thus,\;$d_{\mathcal{F}}(\mu_{n}^{*}, \hat{\mu}_{n}) = O_{P^{*}}\left(n^{-1/2}\right).$
\begin{remark}
Note that~$G$ and~$G_{\mu}$ are not the same gaussian processes!
\end{remark}
We show that $d_{\mathcal{F}}(\mu_{n}^{*}, \hat{\mu}_{n}) = O_{\mathbb{P}_{\nu}}\left(n^{-1/2}\right).$ Indeed, consider the sequence~$S_{n}$ of~order~$O_{P^{*}}(1)$ in~$\mathbb{P}_{\nu}-$probability along the sample, i.e. $$ \lim_{T \rightarrow \infty} \limsup_{n \rightarrow \infty} \mathbb{P}^{*} \{|S_{n}| \geq T \} \rightarrow 0 \text{\;\; in\;} \mathbb{P}_{\nu}- \text{probability along the sample.}$$ Then,
\begin{align*}
 \lim_{T \rightarrow \infty} \limsup_{n \rightarrow \infty} \mathbb{P}_{\nu} \{|S_{n}| \geq T \}  &=  \lim_{T \rightarrow \infty} \limsup_{n \rightarrow \infty} \mathbb{E}_{\nu} \left[ \mathbb{P}^{*} \{ |S_{n}| \geq T \} \right] \\
& \leq \lim_{T \rightarrow \infty} \mathbb{E}_{\nu} \left[ \limsup_{n \rightarrow \infty} \mathbb{P}^{*} \{ |S_{n}| \geq T \}\right] \\ &= \mathbb{E}_{\nu} \left[ \lim_{T \rightarrow \infty} \limsup_{ n \rightarrow \infty} \mathbb{P}^{*} \{ |S_{n}| \geq T \} \right] = 0
\end{align*}
by the dominated convergence theorem and the~Fatou's lemma. Thus, $d_{\mathcal{F}}(\mu_{n}^{*}, \hat{\mu}_{n}) = O_{\mathbb{P}_{\nu}}(n^{-1/2})$ and $d_{\mathcal{F}}(\mu_{n}^{*}, \mu) =  O_{\mathbb{P}_{\nu}}(n^{-1/2}).$
\par Next, we scale~\eqref{eq:ffunctional} by~$\sqrt{n}$:
$$\sqrt{n}(T(\hat{\mu}_{n}) - T(\mu)) = \sqrt{n}(DT_{\mu}(\hat{\mu}_{n} - \mu)) + o_{\mathbb{P}_{\nu}}(1)$$  
and apply Theorem~\ref{uniformclt}. Observe that the linear part in the above equation is gaussian as long as $ 0 < \mathbb{E}_{\mu}T^{(1)}(X_{i}, \mu)^{2} \leq C_{1}^{2}(\mu) \mathbb{E}_{\mu} F^{2}(X)< \infty$ (see Barbe and Bertail (1995), chapter I for details), but that assumption is of course fulfilled since~$\mathcal{F}$ is uniformly bounded. Thus, the following weak convergence in~$l^{\infty}(\mathcal{F})$ holds:
\begin{align*}
\sqrt{n}(T(\hat{\mu}_{n}) - T(\mu)) &= \sqrt{n}(DT_{\mu}(\hat{\mu}_{n} - \mu)) + o_{\mathbb{P}_{\nu}}(1)  \\ &= \sqrt{n} \int_{E} T^{(1)}(x,\mu) (\hat{\mu}_{n} - \mu) d(x)  \\ &= \sqrt{n} \left[ \frac{1}{\hat{n}_{\mathcal{A}_{\mathcal{M}}}} \sum_{i=1}^{\hat{n}_{\mathcal{A}_{\mathcal{M}}}} T^{(1)}(X_{i}, \mu) - 0\right] + o_{\mathbb{P}_{\nu}}(1) \xrightarrow{L} DT_{\mu} G_{\mu}.
\end{align*}
By the previous discussion, we also have 
\begin{align*}
\sqrt{n}(T(\mu_{n}^{*} - T(\mu))) &= \sqrt{n}(DT_{\mu}(T(\mu_{n}^{*} - \mu)) + o_{\mathbb{P}_{\nu}}(1) 
\\ &= \sqrt{n} \int_{E} T^{(1)}(x,\mu) (\mu_{n}^{*} - \mu) d(x) \\ &= \sqrt{n} \left[ \frac{1}{n^{*}_{\mathcal{A}_{\mathcal{M}}}} \sum_{i=1}^{n^{*}_{\mathcal{A}_{\mathcal{M}}}} T^{(1)}(X^{*}_{i}, \mu) - 0\right] + o_{\mathbb{P}_{\nu}}(1).
\end{align*}
The above convergences yield 
\begin{align*} \sqrt{n} [T(\mu_{n}^{*}) - T(\hat{\mu}_{n})] &= \sqrt{n}\left[ \frac{1}{n^{*}_{\mathcal{A}_{\mathcal{M}}}} \sum_{i=1}^{n^{*}_{\mathcal{A}_{\mathcal{M}}}} T^{(1)}(x^{*}_{i}, \mu) - \frac{1}{\hat{n}_{\mathcal{A}_{\mathcal{M}}}} \sum_{i=1}^{\hat{n}_{\mathcal{A}_{\mathcal{M}}}} T^{(1)}(x_{i}, \mu)\right] + o_{\mathbb{P}_{\nu}}(1)  \\ &  \xrightarrow{L} DT_{\mu} G_{\mu}
\end{align*}
and this completes the proof.
\end{proof}
\par Theorem \ref{BootMark} can be easily generalized to the~case when~$\mathcal{F}$ is unbounded and has the envelope in~$L_{2}(\mu)$.
\begin{theorem}\label{BootMarkun}
Let $\mathcal{F}$ be~a~permissible class of functions such that the~envelope~$F$ satisfies
\begin{equation} \mathbb{E}_{A_{\mathcal{M}}} \left[ \sum_{\tau_{A_{\mathcal{M}}} < j \leq \tau_{A_{\mathcal{M}}}(2)} F(X_{j})\right]^{2 + \gamma} < \infty, \;\; \gamma > 0\;(\text{fixed}).
\end{equation}
Suppose, that the~following uniformity condition holds
\begin{equation}
\int_{0}^{\infty} \sqrt{\log N_{2}(\epsilon, \mathcal{F})} d\epsilon < \infty.
\end{equation} Assume further that the conditions of Theorem~\ref{uniformbootun} hold and that $T: \mathcal{P} \rightarrow \mathbb{R}$ is Fr\'{e}chet differentiable functional at~$\mu.$ Then, in general Harris positive recurrent case, we have that $ n^{1/2} (T(\mu_{n}^{*}) - T(\hat{\mu}_{n})) $ converges weakly in~$l^{\infty}(\mathcal{F})$ to~a~gaussian process~$G_{\mu}$ indexed by~$\mathcal{F},$ whose sample paths are bounded and uniformly continuous with respect to the metric~$L_{2}(\mu).$
\end{theorem}
The proof of Theorem \ref{BootMarkun} follows analogously to the proof of~Theorem~\ref{BootMark}. Apply the results of Tsai (1998) and Theorem~\ref{uniformbootun} instead of~Levental's (1988) and Theorem~\ref{uniformclt} to control the remainder terms. Then, the reasoning goes line by line as in the proof of Theorem~\ref{BootMark}.
\begin{remark}
In particular, Theorem \ref{BootMarkun} is also true in the regenerative case.  Replace $\hat{\mu}_{n}$ and $A_{\mathcal{M}}$ by $\mu_{n}$ and $A.$ The proof goes analogously as in the preceding theorems.
\end{remark}
\section{Conclusion}
In this paper, we have shown how the regenerative properties of Markov chains can generalize some concepts in nonparametric statistics from i.i.d. to dependent case. We have shown that uniform bootstrap functional central limit theorem holds over permissible, uniformly bounded classes of functions. We have  proved that the uniform boundedness assumption imposed on~$\mathcal{F}$ can be weakened and it is feasible to require that~$\mathcal{F}$ has an~envelope~in~$L_{2}(\mu).$ We have worked with Markov chains on the general state space, but our results can be directly applied to~Markov chains on~countable state space. Thus, some proofs of the already existing results for the countable case, can be simplified when just applying the methodology introduced in~this paper.
\par The bootstrap asymptotic results for empirical processes indexed by $\mathcal{F}$ naturally lead to bootstrap central limit theorems for~Fr\'{e}chet differentiable functionals.  We have shown that bootstrap uniform CLTs hold in~the~bounded and the~unbounded case over~$\mathcal{F}.$ Similar approach can be also applied to Hadamard differentiable functionals in order to establish analogous asymptotic results to~presented in this paper.
\section*{Acknowledgement}
I would like to thank my advisor, Patrice Bertail, for insightful remarks, inspiring discussions and guidance when I was working on this paper.
\section*{Appendix}
In the small Appendix section we give the short proof of the~Proposition~\ref{radul} which was formulated in~Radulovi\'{c} (2004). We feel the need to provide a~short explanation that this interesting property holds for~any~random variables with~finite second moments. For the~reader's convenience we recall the~Proposition~\ref{radul} below.
\begin{proposition}
For any random variable $W,$ such that $\mathbb{E} W^{2} < \infty,$ there exists a~positive increasing function~$\phi: \mathbb{R}^{+} \rightarrow \mathbb{R}^{+}$ such that
$$ \lim_{x \rightarrow \infty} \frac{\phi(x)}{x^{2}} = + \infty \text{\;\;and\;\;} \mathbb{E} \phi(W) < \infty.$$
\end{proposition} 
\begin{proof}
Consider some positive, increasing function~$\bar{F}(x)$ such that 
$$ \epsilon(x) = \lim_{x \rightarrow \infty} x^{2} \bar{F} (x) = 0.$$ 
\par Firstly, we consider the case, when~$W$ has bounded support. Let~$f$ be a probability density function of~$W.$
For some sufficiently large $x_{0},$ we put
\begin{align*}
\phi(x) &= \left\{ \begin{array}{ll}
\frac{x^{2}}{\epsilon(x)} & \textrm{if $\epsilon(x) \neq 0$}\\
0 & \textrm{else}\\
\end{array} \right. \\
& = \left\{ \begin{array}{ll}
\frac{1}{1-\bar{F}(x)} & \textrm{if $\epsilon(x) \neq 0$}\\
1 & \textrm{else.}\\
\end{array} \right.
\end{align*}
Then, we have 
$$ \lim_{x \rightarrow \infty} \frac{\phi(x)}{x^{2}} = 0 $$
and
\begin{align*}
\int_{x_{0}}^{\infty} \frac{x^{2}}{\epsilon(x)} f(x) dx &= \int_{x_{0}}^{\infty} \phi(x) \bar{F}(dx)  \\
& = \int_{x_{0}}^{\infty} \frac{\bar{F}(dx)}{1-\bar{F}(x)} dx < \infty.
\end{align*}
For the unbounded support case, just put 
for some sufficiently large $x_{0}$ :
\begin{align*}
\phi(x) &= \left\{ \begin{array}{ll}
\frac{x^{2}}{\epsilon(x)} & \textrm{if $\epsilon(x) \neq 0$}\\
C & \textrm{else}\\
\end{array} \right. 
\end{align*}
for some $C > 0,$ then the reasoning is going analogously as in the bounded support case.
\end{proof}

\end{document}